\documentclass{article}

\usepackage{arxiv}

\usepackage[utf8]{inputenc} 
\usepackage[T1]{fontenc}    
\usepackage{hyperref}       
\usepackage{url}            
\usepackage{booktabs}       
\usepackage{amsfonts}       
\usepackage{nicefrac}       
\usepackage{microtype}      
\usepackage{lipsum}

\usepackage{url}

\usepackage{amsthm,amsfonts,amssymb,amsmath,amsbsy,amsthm}
\usepackage{graphicx}
\usepackage{color}
\usepackage{algorithm,algorithmic}
\usepackage{soul}
\usepackage{caption}
\usepackage{subcaption}
\usepackage{lipsum} 
\usepackage{todonotes}
\usepackage{comment}

\newtheorem{theo}{Theorem}

\newcommand{\bea}{\begin{eqnarray}}
\newcommand{\eea}{\end{eqnarray}}

\newcommand{\be}{\begin{equation}}
\newcommand{\ee}{\end{equation}}

\newcommand{\ben}{\begin{equation*}}
\newcommand{\een}{\end{equation*}}

\newcommand{\bean}{\begin{eqnarray*}}
\newcommand{\eean}{\end{eqnarray*}}

\DeclareMathOperator*{\argmin}{arg\,min}

\title{Optimal Experimental Design for Uncertain Systems Based on Coupled Differential Equations}

\date{}

\author{
  Youngjoon Hong\\
 Department of Mathematics\\
 Sungkyunkwan University\\
 Suwon, Republic of Korea\\
  \texttt{hongyj@skku.edu} \\
   \And
 Bongsuk Kwon \\
  Department of Mathematical Sciences\\
  Ulsan National Institute of Science and Technology\\
  Ulsan, Republic of Korea \\
  \texttt{bkwon@unist.ac.kr} \\
    \And
  Byung-Jun Yoon\\
  Department of Electrical and Computer Engineering\\
  Texas A\&M University\\
  College Station, TX, USA\\
  \texttt{bjyoon@ece.tamu.edu} \\ 
}

\begin{document}
\maketitle

\begin{abstract}
We consider the optimal experimental design (OED) problem for an uncertain system described by coupled ordinary differential equations (ODEs), whose parameters are not completely known. The primary objective of this work is to develop a general experimental design strategy that is applicable to any ODE-based model in the presence of uncertainty. For this purpose, we focus on non-homogeneous Kuramoto models in this study as a vehicle to develop the OED strategy. A Kuramoto model consists of $N$ interacting oscillators described by coupled ODEs, and they have been widely studied in various domains to investigate the synchronization phenomena in biological and chemical oscillators. Here we assume that the pairwise coupling strengths between the oscillators are non-uniform and unknown. This gives rise to an uncertainty class of possible Kuramoto models, which includes the true unknown model. Given an uncertainty class of Kuramoto models, we focus on the problem of achieving robust synchronization of the uncertain model through external control. Should experimental budget be available for performing experiments to reduce model uncertainty, an important practical question is how the experiments can be prioritized so that one can select the sequence of experiments within the budget that can most effectively reduce the uncertainty. In this paper, we present an OED strategy that quantifies the objective uncertainty of the model via the mean objective cost of uncertainty (MOCU), based on which we identify the optimal experiment that is expected to maximally reduce MOCU. We demonstrate the importance of quantifying the operational impact of the potential experiments in designing optimal experiments and show that the MOCU-based OED scheme enables us to minimize the cost of robust control of the uncertain Kuramoto model with the fewest experiments compared to other alternatives. The proposed scheme is fairly general and can be applied to any uncertain complex system represented by coupled ODEs.\footnote{This work has been submitted to the IEEE for possible publication. Copyright may be transferred without notice, after which this version may no longer be accessible.}
\end{abstract}

\keywords{Mean objective cost of uncertainty (MOCU) \and optimal experimental design (OED) \and objective uncertainty quantification (objective UQ) \and Kuramoto model}


\section{Introduction}

In many real-world applications, we often have to deal with complex systems, for which we do not have complete knowledge. While collecting more data may lead to better system modeling, there exist many scientific applications in which gathering sufficient data for accurate system identification is practically impossible due to the enormous complexity of the system, prohibitively high cost of data acquisition, or both. Relevant examples abound across various domains, including multi-scale climate modeling for long-term prediction, inference of genome-scale regulatory network for predicting effective intervention strategies, characterization of a material system for optimization of targeted functional properties, just to name a few. In such cases, experimental design should target improving one's knowledge of the uncertain system on aspects that critically affect one's operational goals, be they related to control, classification, filtering, or others.

In this paper, we consider the problem of optimal experimental design (OED) for an uncertain complex system based on coupled ordinary differential equations (ODEs). Specifically, we focus on the Kuramoto oscillator model~\cite{Ku75,Ku03} as a vehicle to develop the OED capabilities, which could be ultimately applied to any uncertain model based on coupled ODEs beyond the Kuramoto model. 
The Kuramoto model has been widely studied by many researchers and has a rich published literature. However, the model has not been much investigated for cases when there is substantial uncertainty in the model.

The primary goal of this paper is to identify the optimal experiment that is expected to effectively reduce the model uncertainty in such a way that minimizes the cost of controlling the uncertain system.
The mean objective cost of uncertainty (MOCU)~\cite{Yoon2013tsp} can be used to quantify the objective-based uncertainty, which then can be used to predict the optimal experiment that maximally reduces the uncertainty pertaining the cost of control~\cite{Dehghannasiri15tcbb}. Recently, MOCU-based OED has been applied to a number of applications such as gene regulatory network intervention~\cite{Dehghannasiri15bmc,Dehghannasiri15tcbb}, adaptive sequential sampling~\cite{Broumand2015pr}, active learning~\cite{Zhao2021WMOCU, Zhao2021SMOCU}, robust filtering~\cite{Zhao2020}, and autonomous materials discovery~\cite{Talapatra2018}. 
While the Kuramoto model has been extensively studied in the past, we would like to note that neither the objective-based uncertainty quantification (UQ) of uncertain Kuramoto models nor OED strategies for reducing the uncertainty present therein have been studied to date.
For the first time, we tackle the experimental design problem for an uncertain coupled-ODE system. 
Unlike previous studies~\cite{Dehghannasiri15tcbb,Dehghannasiri15bmc,Broumand2015pr,Talapatra2018,Zhao2020, ARE2008, Saka88}, we do not assume experiments can directly measure the unknown parameters. 
Instead, we consider realistic experiments, whose outcomes may be used to narrow down the range of the unknown parameters rather than exactly estimating them, which may not be always possible in real applications.


\section{Uncertain Kuramoto Model}

\subsection{Kuramoto model of interacting oscillators}

We consider the Kuramoto model:
\be\label{ku}
\dot \theta_i(t) = \omega_i + \sum_{j=1}^{N} a_{i,j} \sin (\theta_j(t)- \theta_i(t)), \ i=1,\cdots,N,
\ee
where 
$\theta_i = \theta_i(t)$ is the phase of the $i$-th oscillator with the intrinsic natural frequency $\omega_i$, $a_{i,j} (=a_{j,i})$ represents the coupling strength constant between the $i$-th and the $j$-th oscillators, and $N$ is the total number of oscillators in the model. 
The system \eqref{ku} has been first introduced by Yoshiki Kuramoto in~\cite{Ku75, Ku03} to describe the phenomena of collective synchronization observed in the systems of chemical and biological oscillators, in which an ensemble of oscillators spontaneously locks to a common frequency, despite the differences in the natural frequencies of the individual oscillators.
This model makes assumptions that (i) the oscillators are all-to-all, weakly coupled and that (ii) the pairwise interaction between two oscillators depends sinusoidally on their phase difference. 
From a mathematical point of view, the model can be derived by employing the normal form calculation and perturbation method for a system of globally coupled differential equations with stable limits cycles, which is discussed in great details in \cite{Ku03}.
Such collective synchronization phenomena have been observed and investigated in various domains, where circadian rhythms~\cite{L+97}, cortical oscillations in neuroscience~\cite{BHD10, SLK21}, and synchronously flashing fireflies~\cite{Buck76} are well-known examples in biology. There are also examples in engineering and physics, for instance, arrays of lasers~\cite{Jiang:93} and superconducting Josephson junction arrays~\cite{Wie1996}. There have been extensive studies for the Kuramoto model shown in~\eqref{ku}, and we refer interested readers to~\cite{S00, ABRS} for excellent reviews addressing the motivation, derivation, and applications of the model. We also refer to~\cite{DF2011, DF2012, GMT15} and the references therein for other recent theoretical developments.

It is worth noting that the analysis of Kuramoto models in the past, especially their synchronization phenomena, has been made based on mean-field coupling~\cite{ARE2008, Saka88}. Although analytic solutions exist for conditions that guarantee the synchronization of Kuramoto models, they only apply to special cases, such as the homogeneous case when all pairwise coupling strength parameters $a_{i,j}$ are uniform. In our work, we do not constrain our model to such cases, and $a_{i,j}$ can take arbitrary values that are different for different oscillator pairs $(i,j)$. For example, we may even have $a_{i,j}=0$ for some pairs $(i,j)$, such that the interaction graph underlying the Kuramoto model may not be fully connected.

\subsection{Uncertainty class of Kuramoto models}

For Kuramoto models, it is known that the underlying network structure is closely related to synchronization, yet the precise relation is not well understood. Especially we consider the situation where the coupling strengths in the network \eqref{ku} are not fully known.

More precisely, we consider a set of $N$ Kuramoto oscillators, where the natural frequency $\omega_i$ is known for all oscillators. However the interaction strength $a_{i,j}$ between the $i$-th and the $j$-th oscillators is not known with certainty. Instead, we assume that only a lower bound $a_{i,j}^{\ell}$ and an upper bound $a_{i,j}^{u}$ is known for $a_{i,j}$. This gives rise to an uncertainty class $\mathcal{A}$ of $N$ interacting Kuramoto oscillators, where $
\mathbf{a} = \{a_{i,j} \}, \ 1\leq i < j \leq N$
is an uncertain parameter vector. We assume that $\mathbf{a}$ is uniformly distributed in $\mathcal{A}$ following the prior distribution below:
\be
    p(\mathbf{a}) = \begin{cases}
c, & \mbox{if} \ a_{i,j} \in [a_{i,j}^{\ell},a_{i,j}^{u}], \ \forall i,j\\
0, & \mbox{otherwise}
    \end{cases}
    \label{eq:prior}
\ee
where
\be
c=\prod_{i=1}^{N-1}\prod_{j=i+1}^{N} \frac{1}{(a_{i,j}^{u} - a_{i,j}^{\ell})}.
\ee

\subsection{Quantifying the objective cost of uncertainty}


When this uncertain Kuramoto model is initially non-synchronous, can we bring it to synchronization via external control, for example, by connecting an additional oscillator to the rest of the oscillators to catalyze the synchronization? If so, how can it be done optimally?
Suppose we want to ensure the frequency synchronization of the $N$ interacting Kuramoto oscillators with uncertain interaction strength $\mathbf{a}\in\mathcal{A}$ by adding an additional oscillator which interacts with the $N$ oscillators in the original system.
Here the Kuramoto oscillator ensemble $\vartheta(t):=(\theta_1(t), \cdots, \theta_N(t))$ 
is said to achieve \emph{the frequency synchronization asymptotically} if it locks to  a  common frequency such that
\be
    \lim_{t\to\infty} | \dot \theta_i(t) - \dot \theta_j(t) | = 0  \ \ \text{ for all } 1\le i, j \le N.
\ee

We assume that the $(N+1)$-th oscillator added to the original system for control ({\it i.e.}, to achieve frequency synchronization) has a known natural frequency $\omega_{N+1}$ and that its interaction strength with the $i$-th oscillator (part of the original system) is uniform $a_{i,N+1} = a_{N+1}$, $\forall i=1,\cdots,N$. By selecting a sufficiently large interaction strength $a_{N+1}$, we can enforce all $N$ oscillators in the original system to be synchronized with each other in terms of their oscillation frequency (\textit{i.e.}, angular speed). With the introduction of the additional oscillator, now we have
\bea
 \dot \theta_i(t) & = & \omega_i + \sum_{j=1}^{N} a_{i,j} \sin (\theta_j(t)- \theta_i(t)) \nonumber \\ 
 && + a_{N+1} \sin (\theta_{N+1}(t)- \theta_i(t))
\eea
for $i=1,\cdots,N$. Although $a_{N+1}\rightarrow \infty$ would guarantee synchronization, our goal is to minimize the interaction strength $a_{N+1}$ as it affects the cost of control, a larger $a_{N+1}$ resulting in a higher cost for synchronizing the oscillators in the system.

For a given $\mathbf{a}$, we define $\xi(\mathbf{a})$ as the minimum value of the interaction strength $a_{N+1}$ that guarantees synchronization of all oscillators. As $a_{N+1} = \xi(\mathbf{a})$ would be optimal for a specific $\mathbf{a}$, we call it $\xi(\mathbf{a})$ the \textit{optimal interaction strength}. In the presence of uncertainty, we are unable to identify $\xi(\mathbf{a})$ since $\mathbf{a}$ is unknown. Instead, we desire an \textit{optimal robust interaction strength} $\xi^*(\mathcal{A})$ such that
\be
    \xi^*(\mathcal{A}) = \max_{\mathbf{a}\in\mathcal{A}} \xi(\mathbf{a}).   \label{eq:robust}
\ee
Note that it is robust because $a_{N+1} = \xi^*(\mathcal{A})$ guarantees synchronization for any $\mathbf{a}\in\mathcal{A}$. It is optimal because it is the smallest such value.
As we can see from \eqref{eq:robust}, $a_{N+1}$ increases due to the uncertainty, which forces us to choose a larger interaction strength than might be actually needed for synchronization. 
The expected increase of this differential cost can be measured by computing the expected value of the cost increase 
\be
    M(\mathcal{A}) = \mathbb{E}_{\mathbf{a}} \Big[ \xi^*(\mathcal{A}) - \xi(\mathbf{a})\Big],   \label{eq:mocu}
\ee
based on $p(\mathbf{a})$, which governs the distribution of $\mathbf{a}$ within the uncertainty class $\mathcal{A}$. This average differential cost $M(\mathcal{A})$ is referred to as the \textit{mean objective cost of uncertainty (MOCU)}~\cite{Yoon2013tsp}, and it quantifies the impact that the model uncertainty has on the operational objective. When there are two or more objectives, the definition of MOCU in \eqref{eq:mocu} can be further extended as shown in~\cite{Yoon2021multiMOCU}.

\section{Optimal Experimental Design}

Suppose we want to perform additional experiments to reduce the uncertainty class. In general, the outcome of an experiment may reduce the uncertainty class, which may help us predict a better robust controller---in this case, the $(N+1)$-th oscillator with a smaller interaction strength $a_{N+1} = \xi^*(\mathcal{A})$ that ensures the synchronization of all oscillators despite the uncertainty in $\mathbf{a}$. 
While the original oscillators do not necessarily have to be non-synchronous, in such a case the cost of controlling (synchronizing) the Kuramoto model would be zero. To avoid such trivial cases, we assume in our examples that the oscillators in the original model are non-synchronous.
A practical question arises naturally: among the possible experiments, how can we select the optimal experiment? We address this question in what follows.

\subsection{Experimental design space}


We restrict our experimental design space to experiments that test pairwise synchronization between oscillators. Suppose we choose the oscillator pair $(i,j)$ for our experiment, where we initialize the angles to $\theta_i(0) = \theta_j(0)$ and observe whether the two oscillators will become eventually synchronized
in the absence of any influence from all other oscillators. We define a binary random variable $B_{i,j}$ for the experimental outcome, where $B_{i,j}=1$ corresponds to the eventual synchronization of the oscillator pair,  while $B_{i,j}=0$ corresponds to the opposite.

If the interaction strength $a_{i,j}$ between the oscillators $i$ and $j$ were known, the outcome $B_{i,j}$ would be known with certainty. In fact, Theorem~\ref{theo:sync} below shows that the oscillator pair will be synchronized if and only if $|\omega_i - \omega_j| \leq K$, where   $K= 2 a_{i,j}$  in this case.
 \begin{theo} \label{theo:sync}
Consider the Kuramoto model of two-oscillators:
 \begin{equation}\label{Ku2}
 \begin{split}
 \dot \theta_1(t)&= \omega_1 + \frac{K}{2} \sin(\theta_2(t)- \theta_1(t)), \\
  \dot \theta_2(t)&= \omega_2 + \frac{K}{2} \sin(\theta_1(t) - \theta_2(t))
  \end{split}
 \end{equation}
 with the initial angles $\theta_1(0), \theta_2(0)\in[0,2\pi)$.
 Then, for any solutions  $\theta_1$ and $\theta_2$ to \eqref{Ku2}, there holds 
 $|\dot\theta_1(t) - \dot\theta_2(t) | \to 0$ as $t\to\infty$ 
 if and only if
  $|\omega_1-\omega_2|\le K$.
 \end{theo}
 \begin{proof}
First assume that $|\omega_1-\omega_2|\le K$. By symmetry, we may assume without loss of generality, that $\omega_1\ge \omega_2$ and that $\theta_1(0)-\theta_2(0)\in [0,2\pi)$. By subtracting the equations of \eqref{Ku2}, one has
  \begin{equation}\label{sub}
 \dot \Theta(t) = \omega_1-\omega_2 - K \sin(\Theta(t)) =:F(\Theta(t)),
 \end{equation}
 where $\Theta(t):= \theta_1(t) - \theta_2(t)$. 
 Since $0\le \omega_1 -\omega_2 \le K$, there is a number $\Theta_*\in[0,\pi/2]$ such that $\sin \Theta_* = (\omega_1-\omega_2) /K\in[0,1]$. 
 We can easily check that $F'(\Theta_*)=-K \cos \Theta_* \le 0$, indicating that $\Theta_*$ is a unique (mod $2\pi$) stable critical point of \eqref{sub}. This implies that $\Theta(t)\to \Theta_*$ (mod $2\pi$) as $t\to\infty$ unless $\Theta(0)=\pi-\Theta_*$, in turn $\dot\Theta(t)\to0$. When $\Theta(0)=\pi-\Theta_*$, which is an unstable critical point, we have $\Theta(t)=\pi-\Theta_*$ for all $t\ge0$. These prove our assertion. 
 On the other hand, if $|\omega_1-\omega_2|>K$, one has from \eqref{sub} that 
\bea| \dot \theta_1(t)- \dot \theta_2(t) | & \ge & | \omega_1-\omega_2 | - K \left| \sin(\theta_1(t)-\theta_2(t)) \right| \nonumber \\
&\ge & | \omega_1-\omega_2 | - K \nonumber > 0
\eea
This completes the proof. 
 \end{proof}

However, the uncertainty regarding $a_{i,j}$ renders $B_{i,j}$ a random variable whose outcome is unknown before performing the pairwise synchronization experiment described above. Suppose our experiment results in the eventual synchronization of the two oscillators ($B_{i,j}=1$). Based on Theorem~\ref{theo:sync}, the inequality $|\omega_i - \omega_j| \leq 2 a_{i,j}$ must hold.
This experimental outcome allows us to update the lower bound of $a_{i,j}$ from $a_{i,j}^{\ell}$ to $\tilde{a}_{i,j}^{\ell}$ defined as
\be
\tilde{a}_{i,j}^{\ell} = \max(\frac{1}{2}|\omega_i - \omega_j|,a_{i,j}^{\ell}).
\ee
On the other hand, we have $|\omega_i - \omega_j| > 2 a_{i,j}$ if the oscillators do not get synchronized ($B_{i,j}=0$),
which allows us to update the upper bound of $a_{i,j}$ from $a_{i,j}^{u}$ to $\tilde{a}_{i,j}^{u}$ given by
\be
    \tilde{a}_{i,j}^{u} = \min (\frac{1}{2}|\omega_i - \omega_j|,a_{i,j}^{u}).
\ee
In either case, the pairwise experiment can potentially reduce the uncertainty regarding $a_{i,j}$, thereby shrinking the uncertainty class $\mathcal{A}$.

\subsection{Selecting the optimal experiment} \label{s:3_B}


Knowing that the aforementioned pairwise experiments can potentially reduce the uncertainty class, how should we prioritize the experiments to select the optimal one? The MOCU framework can be used to predict the optimal experiment that is expected to maximally reduce the uncertainty~\cite{Dehghannasiri15tcbb,Dehghannasiri15bmc,Broumand2015pr,Zhao2020} in such a way that minimizes the cost of uncertainty, namely, the expected cost increase for controlling (\textit{i.e.}, synchronizing) the $N$ Kuramoto oscillators due to the  uncertain interaction strength. More specifically, for every  experiment in the experimental design space, we first compute the expected remaining MOCU after performing the given experiment. Based on these results, we can prioritize the experiments and select the one that is expected to minimize the MOCU that remains after carrying out the experiment.

For convenience, let us denote the uncertainty class $\mathcal{A}$ reduced based on the experimental outcome $B_{i,j}$ as $\mathcal{A}\vert B_{i,j}$. Then the expected remaining MOCU for the synchronization experiment of the oscillator pair $(i,j)$ can be computed by
\bea
R(i,j) & = & \mathbb{E}_{B_{i,j}} \Big[ M(\mathcal{A}\vert B_{i,j})\Big] \nonumber\\ 
& = & \sum_{b\in\{0,1\}} Pr(B_{i,j}=b)M(\mathcal{A}\vert B_{i,j}=b)
\label{eq:remaining_mocu}
\eea
Based on the prior $p(\mathbf{a})$ in \eqref{eq:prior}, we can compute the probabilities for the possible experimental outcomes as follows
\bea
    Pr(B_{i,j}=1) & = & \frac{a_{i,j}^{u}-\tilde{a}_{i,j}}{a_{i,j}^{u}-a_{i,j}^{\ell}}, \\
    Pr(B_{i,j}=0) & = & \frac{\tilde{a}_{i,j}-a_{i,j}^{\ell}}{a_{i,j}^{u}-a_{i,j}^{\ell}}
\eea
where we define 
\be \label{e:a_tilde}
\tilde{a}_{i,j} = \min \left (\max  \left(\frac{1}{2}|\omega_i - \omega_j|,a_{i,j}^{\ell} \right), a_{i,j}^{u} \right )
\ee
to ensure that $\tilde{a}_{i,j} \in [a_{i,j}^{\ell},a_{i,j}^{u}]$.
The optimal oscillator pair, the outcome of whose pairwise synchronization experiment is expected to most effectively improve the control performance among the  $N \choose 2$ pairs can be predicted by
\be
(i^*, j^*) = \argmin_{1\leq i < j \leq N} R(i,j).
\ee

\subsection{Computational complexity} \label{s:3_B}

While computing the entropy of an uncertain coupling strength takes a fixed amount of time  $\mathcal{O}(1)$, the computation of MOCU in \eqref{eq:mocu} requires solving the ordinary differential equation in \eqref{ku} multiple times to estimate the expectation based on Monte Carlo sampling. As a result, the computational complexity of computing the MOCU defined in \eqref{eq:mocu} is
\be
    \mathcal{O}(T S N^2),   \label{eq:mocu_complexity}
\ee
where $T$, $S$, and $N$ denote the final duration of time for solving the differential equation, the sample size (for Monte Carlo sampling), and the number of oscillators in the Kuramoto model, respectively. For practical values of these parameters, please see Section~\ref{s:4}. However, we would like to note that the computation of MOCU does not take a large amount of time in practice, since the main MOCU computation module can be highly parallelized by taking advantage of  modern GPU programming. In fact, we implemented our code using \texttt{PyCUDA}~\cite{PyCUDA12} and ran our simulations with \texttt{Nvidia GTX 1080-Ti}, which reduces the overall computational cost by $\frac{1}{n}$, where $n= 57,344$ is the maximum number of CUDA cores in the graphic card. This massive parallelization allows us to reduce the computational cost significantly by increasing the number of threads on the thread block. For instance, the average runtime for the MOCU computation module was about $2.87$ seconds when $N=5$, $T=4$, and $S = 20,000$.


\section{Simulation Results}\label{s:4}


In this section, we present numerical experiments to demonstrate our proposed OED method described in Section~\ref{s:3_B}. 
A classical fourth-order Runge-Kutta method is used to compute the Kuramoto model in \eqref{ku} for $0 \leq t \leq T$ for $T=5$ with time discretization $\Delta t = 1/160$.
For the sake of simplicity, the initial conditions are set to $\theta_i = 0$, $1 \leq i \leq N$. 
As the system is regular enough, the Runge-Kutta solvers provide reasonably accurate numerical solutions.
For instance, the relative $L^2$-error of a five-oscillator model ($N=5$) is $10^{-10}\sim10^{-9}$ at $t=T$.
Due to the immense uncertainty in parameters, the sample size $S$ for computing the expectation in \eqref{eq:mocu} should be sufficiently large (e.g. $S \geq 20,000$) to obtain reliable experimental results.
In this regard, massive computing power is highly desired, and we adopt GPU parallel computing, \texttt{PyCUDA}~\cite{PyCUDA12}, using \texttt{Nvidia GTX 1080-Ti}.

\begin{figure}
    \centering
    \includegraphics[scale=0.75]{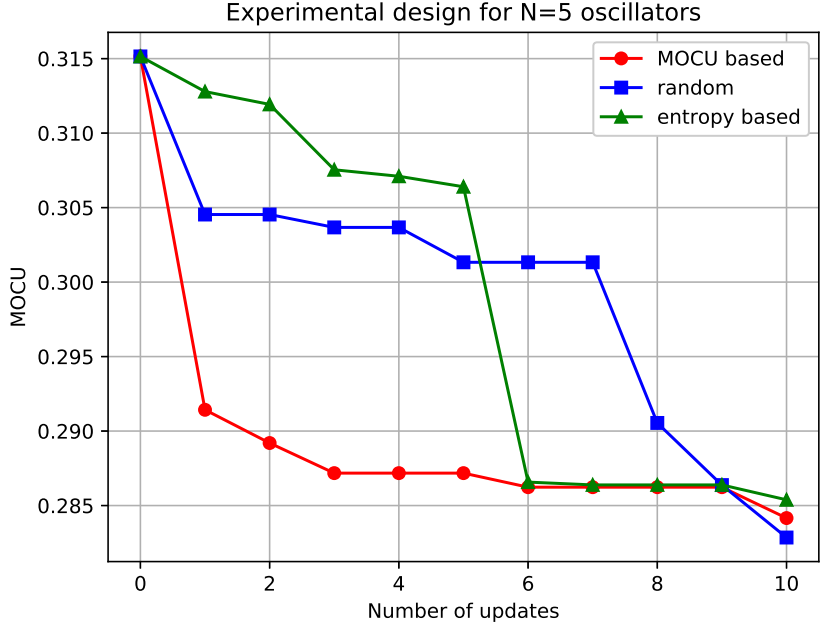}
    \caption{The model uncertainty decreases with sequential experimental updates. The MOCU-based scheme is clearly the most efficient, quickly reducing the uncertainty in fewer updates.}
    \label{fig:5_sample}
\end{figure}

As a paradigm example, we first implemented a 5-oscillator Kuramoto model.
The following natural frequencies were used for the experiment in Figure~\ref{fig:5_sample}: $w_1 = -2.5000, w_2 = -0.6667, w_3 = 1.1667, w_4 = 2.0000, w_5 = 5.8333$.
For the additional oscillator used for control, we simply chose $w_6 = \text{mean}_{1 \leq i \leq 5}w_i$.
The upper and lower bounds of interaction strength were chosen by
\ben
\begin{split}
a_{i,j}^u & = 1.15 d_{i,j} \left(\frac{1}{2}|w_i - w_j|\right),\\
a_{i,j}^l & = 0.85 d_{i,j} \left(\frac{1}{2}|w_i - w_j|\right)
\end{split}
\een
where $d_{i,j}$ is a correction constant.
If $d_{i,j} \equiv 1$ for all $i,j$, the system is already synchronized in general as all entries of the interaction strength $a_{i,j}$ are large enough. 
Hence, we introduced the correction parameter $d_{i,j}$ such that
\ben
d_{i,j} = 
\begin{cases}
1, \quad (i,j) \in \mathcal{I}_1,\\
d^*_{i,j}, \quad (i,j) \in \mathcal{I}_2,\\
\end{cases}
\een
where $0.3 \leq d^*_{i,j} \leq 0.5$, and $\mathcal{I}_1$ and $\mathcal{I}_2$ are partition of the set of indices $\mathcal{I} = \{(i,j) \in \mathbb{N}: 1 \leq i,j \leq N, i < j \}.$
Here, the set $\mathcal{I}_1, \mathcal{I}_2$ and the corresponding quantity $d^*_{i,j}$ were empirically determined. This results in a class of uncertain Kuramoto models that are non-synchronous, fully connected, but whose coupling strengths are uncertain.

To compare numerical results, we carried out a sequence of experiments based on three different experiment selection strategies: MOCU-based, random selection, and entropy-based.
In the MOCU-based selection strategy, the pairwise experiment with the smallest expected remaining MOCU in \eqref{eq:remaining_mocu} was chosen, and the corresponding entry of $a^u_{i,j}$ or $a^\ell_{i,j}$ was updated based on the experimental outcome.
More precisely, from the result of the pairwise experiment between the $i$-th and the $j$-th oscillators, if the oscillators were synchronized, then the lower bound $a^\ell_{i,j}$ was updated to $\tilde a_{i,j}$ defined in \eqref{e:a_tilde}. Otherwise, the upper bound $a^u_{i,j}$ was updated to $\tilde a_{i,j}$.
In the entropy-based approach, we selected the pairwise experiment with the largest value of $a^u_{i,j} - a^\ell_{i,j}$, and the corresponding entry of $a^u_{i,j}$ or $a^\ell_{i,j}$ was updated based on the outcome in the same fashion.
In the random approach, we randomly chose one of the $N \choose 2$ possible experiments and updated the corresponding entry of $a^u_{i,j}$ or $a^\ell_{i,j}$ based on the experimental outcome.
Figure~\ref{fig:5_sample} shows that while all three methods eventually reached the minimum attainable uncertainty level after exhausting all experiments, the MOCU-based experiment selection strategy led to the most efficient updates. Especially, it identified effective experiments early on, nearly reaching the minimum attainable uncertainty level in just 3 updates.

\begin{figure}
    \centering
    \includegraphics[scale=0.75]{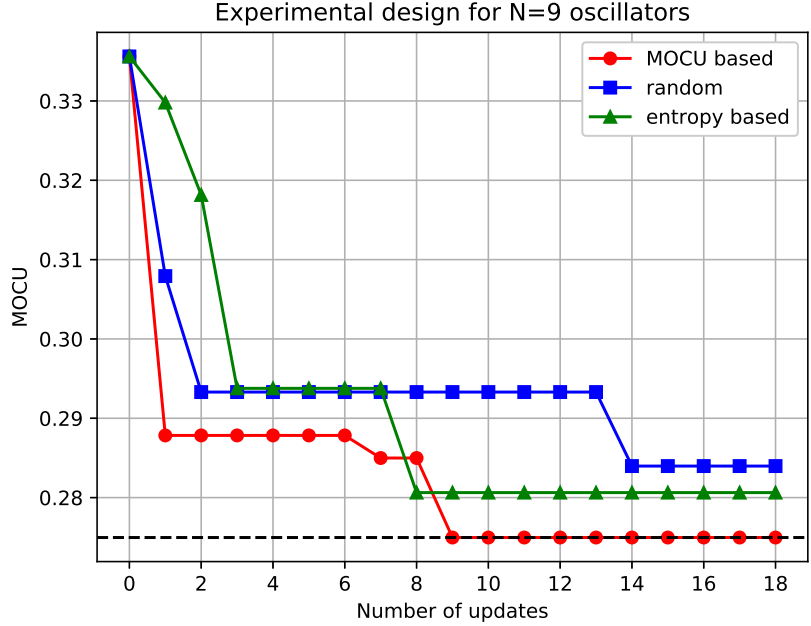}
    \caption{The model uncertainty decreases with sequential experimental updates. The baseline shows the minimum attainable uncertainty after all possible updates (36 experiments in total).}
    \label{fig:9_sample}
\end{figure}

Next, we conducted numerical simulations with a larger number of oscillators, $N=9$, in which case the number of possible experiments increases to ${9 \choose 2} = 36$.
We used the following natural frequencies: $w_1 = 1.19, w_2 = 3.23, w_3 = 6.34, w_4 = 7.48, w_5 = 10.9, w_6 = 11.62, w_7 = 14.74, w_8 = 29.58, w_9 = 38.88.$
We set the natural frequency of the additional oscillator to $w_{10} = \frac{1}{2}\text{mean}_{1 \leq i \leq 9}w_i$.
To start with a non-synchronized model, we selected sufficiently large natural frequencies for $w_8$ and $w_9$.
The upper and lower bounds of the interaction strengths were set to
\ben
\begin{split}
a_{i,j}^u & = 1.03 d_{i,j} \left(\frac{1}{2}|w_i - w_j|\right), \\ 
a_{i,j}^l & = 0.97 d_{i,j} \left(\frac{1}{2}|w_i - w_j|\right)
\end{split}
\een
where $d_{i,j}$ is the correction parameter and that was empirically determined as in the previous example.
Figure~\ref{fig:9_sample} shows the simulation results, which clearly demonstrate that the MOCU-based experimental design results in the most efficient updates among all three methods.
Here we considered up to 18 updates (out of 36 experiments in total). As shown in Figure~\ref{fig:9_sample}, the MOCU-based OED was able to drastically reduce uncertainty with a single update, and it reached the minimum uncertainty level just in 9 updates.


\section{Concluding Remarks}

As shown in our results, designing effective experiments for complex uncertain systems requires quantifying the state of our current knowledge of the system and measuring  the impact of the remaining uncertainty  on the operator performance. It is clear that we cannot expect system-agnostic black-box optimization schemes to perform well. Furthermore, experiments that aim to enhance our knowledge regarding parameters with the largest uncertainties do not necessarily help, as they may not be pertinent to the operator performance. Our work shows that the MOCU-based OED framework can effectively prioritize the experiments in the design space by quantifying the impact of their potential outcomes on the operational goal using scientific knowledge as in Theorem~\ref{theo:sync}.

The proposed OED strategy can be immediately applied to various scientific problems that utilize Kuramoto models for studying synchronization phenomena in diverse fields. For example, Kuramoto models have been widely used for investigating brain network synchronization~\cite{Kitzbichler2009, Choi2019} and its relation to neurological disorders~\cite{Mohseni2017}. Different brain regions may be represented by different oscillators in a Kuramoto model, where the coupling strengths reflect the underlying neuronal connectivity between regions. The resulting Kuramoto model is well-known for its capability to effectively capture oscillatory brain dynamics~\cite{Kitzbichler2009, Choi2019}. The graph structure underlying the Kuramoto model and the actual coupling strength may not be known with certainty and they may have to be estimated from neuroimaging data~\cite{De2014, Wang2011} or through experiments that combine imaging techniques with noninvasive brain stimulation~\cite{Shafi2012}. The uncertain Kuramoto model may be used for various tasks, including the prediction of a personalized structure–function relationship~\cite{Bansal2018} or therapeutic modulation of brain activity for management of neurological disorders~\cite{Shafi2012, Mohseni2017}.

Beyond the Kuramoto models, on which we focused in this study, our MOCU-based OED strategy can be adapted to a wide range of other uncertain systems that are described by ordinary differential equations. For example, one may apply the proposed OED scheme to the effective uncertainty reduction in non-linear systems considered in~\cite{He2017, He2018, Lyu2020}.

An interesting direction for future research is to expand the experimental design space to encompass experiments, whose outcomes cannot be directly used to measure the uncertain model parameters (as assumed in many past studies~\cite{Dehghannasiri15tcbb,Dehghannasiri15bmc,Broumand2015pr,Talapatra2018,Zhao2020}) or to
reduce model uncertainty by leveraging mathematical theorems (as we have proposed in this work). In general, we expect that taking a Bayesian approach would be most effective for addressing this inverse problem~\cite{Stuart2010, Ray2013, Adler2018}, and we are currently investigating strategies for Bayesian inversion of experimental outcomes to effectively reduce the objective model uncertainty.

\bibliographystyle{unsrt}

\bibliography{references.bib}

\end{document}